\documentclass [11pt,reqn]{amsart}

\usepackage{graphicx}
\usepackage{indentfirst}
\usepackage[mathscr]{eucal}
\usepackage[all]{xy}
\usepackage{amsmath,amsxtra,amssymb,latexsym,amscd,amsthm,wasysym, amsfonts,color,esint,mathtools,geometry,cite,hyperref}
\usepackage{fancyhdr}
\usepackage{dsfont}
\usepackage{comment}
\usepackage{enumitem}
\usepackage{inputenc}
\hypersetup{
    unicode=false,        
    pdftoolbar=true,      
    pdfmenubar=true,       
    pdffitwindow=false,     
    pdfstartview={FitH},    
    pdftitle={},    
    pdfauthor={},     
    colorlinks=true,       
   linkcolor=blue,          
    citecolor=blue,        
    filecolor=black,      
    urlcolor=black} 

\title[Smoothness of bounded solutions to complex Hessian equations]{Smoothness of bounded solutions to complex Hessian equations on compact Hermitian manifolds}
\author{Yuetong Fang}

\address{Université d'Angers, CNRS, LAREMA, SFR MATHSTIC, F-49000 Angers, France}
\email{yuetong.fang@univ-angers.fr}

\begin{document}

\begin{abstract}
  We establish an a priori $C^2$-estimate for complex Hessian equations on compact Hermitian manifolds. As an application, we obtain a regularity result for solutions to complex Hessian equation with densities of the right hand side depending on the solution. This generalizes Sz\'ekelyhidi--Tosatti's result to the setting of complex Hessian equations on Hermitian manifolds. 
\end{abstract}

\maketitle

\def\PSH{\operatorname{PSH}}
\def\SH{\operatorname{SH}}
\def\ddbar{\frac{\sqrt{-1}}{2\pi}\partial\bar{\partial}}
\def\ddb{\sqrt{-1}\partial\bar{\partial}}
\def\tmax{T_{\max}}
\def\dd{\partial}
\def\ddc{dd^c}
\def\omt{\omega_0+\frac{\sqrt{-1}}{2\pi}\partial\bar{\partial}\varphi}
\def\p{\varphi}
\def\omdt{\hat{\omega}_t+\frac{\sqrt{-1}}{2\pi}\partial\bar{\partial}\varphi}
\def\tr{\mathrm{tr}}
\def\Capa{\mathrm{Cap}}
\def\MA{\mathrm{MA}}
\def\ric{\mathrm{Ric}}
\def\Vol{\mathrm{Vol}}
\def\osc{\mathrm{Osc}}
\def\dtdelta{\dfrac{\partial}{\partial t}-\Delta} 
\def\pdot{\dot{\varphi}}
\numberwithin{equation}{section}
\def\1{\mathds{1}}
\newcommand{\A}{\mathbb{A}}
\newcommand{\B}{\mathbb{B}}
\newcommand{\C}{\mathbb{C}}
\newcommand{\G}{\mathbb{G}}

\newcommand{\Gm}{\mathbb{G}_\mathrm{m}}

\newcommand{\N}{\mathbb{N}}
\renewcommand{\P}{\mathbb{P}}
 \newcommand{\Q}{\mathbb{Q}}
 \newcommand{\R}{\mathbb{R}}
 \newcommand{\Z}{\mathbb{Z}}
\newcommand{\T}{\mathbb{T}}

\def\PSH{\operatorname{PSH}}
\def\SH{\operatorname{SH}}
\def\ddbar{\frac{\sqrt{-1}}{2\pi}\partial\bar{\partial}}
\def\ddb{\sqrt{-1}\partial\bar{\partial}}
\def\tmax{T_{\max}}
\def\dd{\partial}
\def\ddc{dd^c}
\def\omt{\omega_0+\frac{\sqrt{-1}}{2\pi}\partial\bar{\partial}\varphi}
\def\p{\varphi}
\def\omdt{\hat{\omega}_t+\frac{\sqrt{-1}}{2\pi}\partial\bar{\partial}\varphi}
\def\Tr{\mathrm{Tr}}
\def\Capa{\mathrm{Cap}}
\def\MA{\mathrm{MA}}
\def\ric{\mathrm{Ric}}
\def\Vol{\mathrm{Vol}}
\def\osc{\mathrm{Osc}}
\def\dtdelta{\dfrac{\partial}{\partial t}-\Delta} 
\def\pdot{\dot{\varphi}}
\numberwithin{equation}{section}
\def\1{\mathds{1}}

\newtheorem{theorem}{Theorem}[section]
\newtheorem{thm}[theorem]{Theorem}
\newtheorem{lemma}[theorem]{Lemma}
\newtheorem{lem}[theorem]{Lemma}
\newtheorem{prop}[theorem]{Proposition}
\newtheorem{coro}[theorem]{Corollary}
\newtheorem{corollary}[theorem]{Corollary}
\newtheorem*{Main Theorem}{Main Theorem}

\theoremstyle{definition}
\newtheorem{definition}[theorem]{Definition}
\newtheorem{defi}[theorem]{Definition}
\newtheorem{example}[theorem]{Example}
\newtheorem{exa}[theorem]{Example}
\newtheorem{claim}[theorem]{Claim}
\newtheorem{remark}[theorem]{Remark}
 \newtheorem*{ackn}{Acknowledgments}

\section{Introduction}
Complex Hessian equations constitute an important class of fully nonlinear elliptic equations arising naturally in complex geometry. They generalize the complex Monge--Amp\`ere equations, corresponding to the case $m=n$, and have been extensively studied in recent years. Let $(X,\omega)$ be a compact Hermitian manifold. Fix $1 < m \le n $. We study the complex Hessian equation
\begin{equation} \label{Eqn: Hessian}
    (\omega + \ddc \varphi )^m \wedge \omega^{n-m } = e^{f(\cdot,\varphi )} \omega^n,
\end{equation}
where $f: X \times \R \rightarrow \R$ is a smooth function. 

When $m=n$, the equation \eqref{Eqn: Hessian} has numerous applications in geometry. When $e^{f(x,\varphi)}= e^{f(x)}$, Yau’s solution \cite{Yau78} to the Calabi's Conjecture has motivated many subsequent studies on Monge--Ampère equations over the past fifty years. When $f$ depends on the unknown and it is not increasing, the existence of solutions is not guaranteed and even if they exist there may be no uniqueness. Apart from the continuity method, there are many other method which often provide weak solutions. Thus a natural question is whether solutions inherit, in some sense, regularity properties of the data. Using regularizing properties of the K\"ahler--Ricci flow,  Sz\'ekelyhidi and Tosatti \cite{Sz_kelyhidi_2011} proved that any bounded solution to \eqref{Eqn: Hessian} is smooth. Consequently, K\"ahler--Einstein currents on Fano manifolds are smooth. In the Hermitian setting, when $f$ is non-decreasing in the unknown, smooth solutions to \eqref{Eqn: Hessian} exist thanks to work of Tosatti--Weinkove \cite{TW10} culminating considerable efforts from many others (\cite{GL10}, \cite{Cher87}, etc).  In the general case (i.e. when $f$ is not increasing in $t$), the method of Sz\'ekelyhidi--Tosatti has been adapted to the Hermitian setting by Nie.  Assuming the invariance of the Monge-Amp\`ere volume,  Nie \cite{Nie2014Regularity} proved that any bounded solution to equation \eqref{Eqn: Hessian} ($m=n$) is smooth. She used  regularizing properties of the Chern--Ricci flow; and for this she needed a stability estimate which requires the condition on invariance of mass. As shown by Chiose \cite{Chi16}, the latter condition is equivalent to a curvature assumption introduced earlier by Guan--Li \cite{GL10}. The stability estimate required in the proof of Nie has been later on provided by Ko{\l}odziej--Nguyen in \cite{KNstabilityMA} without any curvature assumption; the regularity of solutions when $m=n$ is thus ensured in full generality.


We now turn to the general complex Hessian equations when $f(x,\varphi)=f(x)$. 
In the K\"ahler case, the existence of a unique smooth solution to \eqref{Eqn: Hessian} has been extensively studied by many authors (see \cite{HouHessian}, \cite{HMW10}, etc) with final resolution by Dinew--Ko{\l}odziej \cite{LiouvilleDK17}.  
In the Hermitian setting, Zhang \cite{Zhang_2017} and Sz\'ekelyhidi \cite{Sze18} independently proved existence of smooth solutions by two different a priori second order estimates.  

Adapting Sz\'ekelyhidi's work \cite{Sze18}, Ko{\l }odziej and Nguyen \cite{KN15weaksol} proved a regularity result for the weak solution in the case where $e^{f(x, \varphi )} = e^{\varphi +f(x)}$, and used it to obtain an approximation theorem for $(\omega,m)$-subharmonic functions.

In the general case when $f$ may be decreasing in the unknown, as in the Monge--Amp\`ere case, existence of uniqueness of solutions may fail. In \cite{sheng2021}, on K\"ahler manifolds, the authors proved the smoothness of bounded solutions using the parabolic method, inspired by \cite{Sz_kelyhidi_2011}. 

The pluripotential theory for complex Hessian operators has been well studied in the recent years in K\"ahler settings \cite{DK14,LU2013785,DinewLu2015,LN15}. In Hermitian settings, potential theory for $(\omega,m)$-subharmonic functions has been well developed (see \cite{KN16,GuNguyen2018,guedj2023degeneratecomplexhessianequations,KN25Holder,KN25polarset,sun2025}, etc). It follows from \cite{kolodziej2023complexhessianmeasuresrespect} that the Hessian operators for bounded $(\omega,m)$-subharmonic functions are well defined on Hermitian manifolds. It is thus natural to consider higher regularity of solutions to \eqref{Eqn: Hessian}. 

In this paper, we study the regularity of the equation \eqref{Eqn: Hessian} in the Hermitian setting. More precisely, our main result is the following: 
\begin{Main Theorem}
Let $(X,\omega)$ be a compact Hermitian manifold. Fix $2 \le m \le n$. Let $\varphi$ be a bounded $(\omega,m )$-subharmonic function solving the complex Hessian equation \eqref{Eqn: Hessian}, where $f: X \times \R \rightarrow \R$ is a smooth function. Then $\varphi$ is smooth.
\end{Main Theorem}
Our proof is inspired by \cite{BLZ26}, where they established uniform estimates of the Monge--Amp\`ere equations:
\[
(\omega +\ddc u )^n = e^{u -\varphi + f(x,\varphi)},
\]
for functions $\varphi$ smooth and uniformly bounded with estimates depending only on the $L^{\infty}$-norm of $\varphi$. 

We now outline the main idea of the proof. Approximate the bounded solution $\varphi$ of the equation \eqref{Eqn: Hessian} by a sequence of smooth $(\omega,m)$-subharmonic function $v_j$, and consider the complex Hessian equation
\begin{equation} \label{eqn: Hessian 2}
(\omega + \ddc u_j )^m \wedge \omega ^{n-m } = e^{u_j - v_j + f(x,v_j )} \omega^n.
\end{equation}
Assume that $u_j$ converges to a function $u \in \SH_m(X,\omega )$. The domination principle yields that $u = \varphi$. The main difficulty in this approach is to obtain an a priori $\mathcal{C}^2$-estimate independent of $|\nabla v_j |^2$ and $\Delta v_j$.
However, the argument developed for the complex Monge--Amp\`ere equation cannot be directly adapted to the complex Hessian equation. The main difficulty is that an admissible solution may have negative eigenvalues of its complex Hessian. In view of this difficulty, we follow the approach of Hou–Ma–Wu \cite{HMW10} and Székelyhidi \cite{Sze18} and use the largest eigenvalue of the complex Hessian, rather than its trace, in the construction of the test function for the maximum principle. 

Since the $\mathcal{C}^2$-estimate for equation \eqref{eqn: Hessian 2} need to be independent of $|\nabla v_j |^2$ and $\Delta v_j$, extra work are needed. Here, we follow the idea of \cite{BLZ26}, adding an extra term $\chi(v)$ to the test function so as to eliminate the dependence on $\nabla v$ and $\Delta v$. Comparing to the Monge--Amp\`ere equation, the main difficulty in the Hessian setting is that the eigenvalues of the Hessian matrix of $v$ are not necessarily all positive. This lack of positivity gives rise to additional negative terms in the second-order estimate, which require further control.

A further difficulty arises from the structure of the linearized operator. In the maximum-principle argument, we are led to control the quantity $F^{pp } G_{p \bar{p} }$ rather than the trace of the test function in Monge--Amp\`ere setting, where $F^{pp}$ denotes the derivative of the equation with the $pp$-entry. Since $F^{pp}$ involve elementary symmetric polynomials, the estimate requires delicate control of such polynomials. In particular, suitable tools for elementary symmetric polynomials are essential to the argument. 

Another difficulty of adapting the $\mathcal{C}^2$-estimate for Monge--Amp\`ere equation in \cite{BLZ26} to complex Hessian equations is that, the $\mathcal{C}^2$-estimate for complex Hessian equations depends on the gradient estimate. It is essential to use the identity $G_p =0$ to deal with the $\Phi(|\nabla u |^2 )$ term. As a consequence, the $-\chi'^2|v_p|^2$ term arises. Since $\chi'$ is chosen to be large enough to cancel the negative $v_{p\bar{p}}$ terms, it is delicate to vanish the negative $\chi'^2$ term. Therefore, instead of considering $\chi(v)$, we consider $\chi(v-\tilde{\varphi})$, where $\tilde{\varphi}$ is a smooth function such that $|v-\tilde{\varphi}|\le \epsilon$, for uniform constant $\epsilon$ small enough. This construction allows us to use part of $\chi''$ term to cancel $\chi'^2$ term.
\subsection*{Acknowledgment}
 The author thanks her supervisor, Hoang-Chinh Lu, for suggesting the problem, valuable discussions during the preparation of the work, and useful comments. This research is part of a PhD program funded by the PhD scholarship CSC-202308070110, and is partially supported by Institut Universitaire de France, Centre Henri Lebesgue ANR-11-LABX-0020-01, and the KRIS project of Fondation Charles Defforey. 

\subsection*{AI declaration} This research and manuscript were completed without the use of any artificial intelligence.

\subsection*{Organization of the paper}The paper is organized as follows. Section~\ref{sec:prelim} gives preliminary material on elementary symmetric polynomial and $(\omega,m )$-subharmonic functions. We also develop some inequalities for $m$-positive forms. In section~\ref{sec:est}, we establish a $\mathcal{C}^2$-estimate for equation \eqref{eqn: Hessian 2}. We prove the main result of this paper in Section~\ref{sec:mainthm}.

\section{Preliminaries} \label{sec:prelim}
\subsection{Symmetric functions and Gårding cones}
\subsubsection{Elementary symmetric polynomials}
Let $1\le m<n$. The Gårding cone associated with the $m$-th elementary symmetric function is defined by
\[
\Gamma_m = \{\lambda\in\mathbb R^n: S_1(\lambda)>0, \cdots ,S_k(\lambda)> 0,\ k=1,\ldots,m\},
\]
where
\[
S_k(\lambda) = \sum_{1\le i_1<\cdots<i_k\le n} \lambda_{i_1}\cdots\lambda_{i_k}.
\]
For $\{i_1 ,\cdots,i_t \} \subset \{1,\cdots,n \} $, we denote
\[
S_k(\lambda; i_1 ,\cdots,i_t ) \coloneqq S_k|_{ \lambda_{i_1}= \cdots = \lambda_{i_t}=0} .
\]
It is clear that the following splitting formula holds:
\begin{equation} \label{for:split}
S_m (\lambda) = S_{m}(\lambda ; i) + \lambda_i S_{m-1}(\lambda;i).
\end{equation}
We also recall the identity 
\begin{equation} \label{for:sumsplit}
    \sum_p \lambda_p S_m(\lambda;p) = (m+1)S_{m+1}(\lambda).
\end{equation}
We briefly recall some properties related to elementary symmetric polynomials that will be used later.
\begin{prop} \label{prop:espstandard}
    The following inequalities hold:
    \begin{enumerate}
        \item (Maclaurin's inequality) If $\lambda \in \Gamma_m$, then $\left( \frac{S_j}{\binom{n}{j}} \right)^{1/j} \le \left( \frac{S_i}{\binom{n}{i}} \right)^{1/i}$ for $1 \le i \le j \le m$.
        \item (Newton's inequality) If $k \ge 2$, $\lambda \in \R^n$, then 
        \begin{equation} \label{ine:Newton}
        S_k^2 \ge S_{k-1}S_{k+1} \frac{\binom{n}{k}^2} {\binom{n}{k-1}\binom{n}{k+1}}.
        \end{equation}
        \item  (Generalized Newton--Maclaurin inequality) For $\lambda \in \Gamma_m$,
        \[
 \left(\frac{S_k(\lambda)}{S_l(\lambda)} \right)^{\frac{1}{k-l} } \le \left(\frac{S_r(\lambda)}{S_s(\lambda)} \right)^{\frac{1}{r-s} },
\]
     where $0 \le s <r, 0 \le l <k, r<k, s<l$.
     \item (G\aa rding's inequality) Let $M$ be the polarization form of $S_m$. Then for all $x^{1}, \cdots,x^m \in \Gamma_m $
     \[
    M(x^1, \cdots, x^m ) \ge S_m(x^1)^{\frac{1}{m} } S_m(x^2)^{\frac{1}{m} } \cdot \cdots \cdot S_m(x^m)^{\frac{1}{m} } .
    \]
     \item For all $\lambda \in \Gamma_m$, and $j = 1,\cdots, n-1$,
     \[
     \frac{\partial S_m(\lambda)}{\partial \lambda_j}(\lambda) >0.
     \]
     Moreover, summing up $ \frac{\partial S_2(\lambda)}{\partial \lambda_j}(\lambda) $ over index $j$, we obtain
     \[
     \lambda_j \ge - (n-1 ) S_1.
     \]
     Together with $\lambda_p \le S_1$, we obtain 
     \begin{equation} \label{frml:valerabsolupse}
     |\lambda_j | \le (n-1) S_1,
      \end{equation}
     for all $\lambda \in \Gamma_m$, with $m\ge 2$.
    \end{enumerate}
\end{prop}
We refer the reader to \cite{gar59,CNSIII85} for proofs of the above proposition and for further properties of the elementary symmetric polynomials.

\subsubsection{$m$-positive cones associated with differential forms} Let $\Omega$ be a bounded open set in $\C^n$. Given a Hermitian metric $\omega$ on $\Omega$, a real $(1,1)$-form $\alpha$ on $\Omega$ is said to be \emph{$m$-positive with respect to $\omega$} if
\[
\alpha^k \wedge \omega^{n-k} \ge 0, \qquad k=1,\ldots,m,
\]
at every point of $\Omega$.

Equivalently, at any point $P\in\Omega$, in local coordinates in which $\omega$ is represented by the identity and $\alpha$ is diagonalized as
\[
\alpha = \sqrt{-1}\sum_{j=1}^n a_j\, dz_j \wedge d\bar z_j,
\]
the vector $a=(a_1,\ldots,a_n)$ belongs to the closure the cone $\Gamma_m$.
The positive cone $\Gamma_m(\omega)$ associated with the metric $\omega$ is defined by 
\[
\left\{  \gamma \text{ real }(1,1) \text{-form}: \gamma^k\wedge \omega^{n-k}>0, k=1,\cdots,m \right\}.
\]
It follows from Gårding's inequality \cite{gar59} that if $\gamma_0, \gamma_1, \cdots, \gamma_{m-1}\in \Gamma_{m}(\omega)$, then 
\[
\gamma_0 \wedge\gamma_1\wedge \cdots \wedge\gamma_{m-1} \wedge\omega^{n-m} >0. 
\]
In this subsection, we prove an optimal lower bound for  $\alpha^{m-1} \wedge \gamma \wedge \omega^{n-m}$.
\begin{lemma} \label{lem: elemhyppol}
Let $\alpha$ be a $m$-positive form with respect to $\omega$ satisfying $S_m(\alpha) >0$. Then there exists a constant $\epsilon>0$, depending on the elementary symmetry polynomials of $\alpha$, such that for every  $1 \le k \le m$,
\[
\alpha - \epsilon \omega\in \Gamma_k(\omega) .
\]
Moreover, the constant $\epsilon$ can be written as $\epsilon = \frac{1}{k} \frac{\sigma_{k}(\alpha)}{\sigma_{k-1}(\alpha)}$, where $\sigma_j(\alpha)$ denotes the normalized $j$-th symmetric polynomial
\[
\sigma_j( \alpha) = \frac{\alpha^j \wedge\omega^{n-j }}{\omega^n }.
\]
\end{lemma}

\begin{proof}
 Expanding $(\alpha - \epsilon_k\omega )^k \wedge \omega^{n-k}$, we obtain
 \begin{equation} \label{eqn:expandmpositive}
 \begin{aligned}
  \frac{(\alpha - \epsilon_k\omega )^k \wedge \omega^{n-k}}{\omega^n} = \frac{\sum_{j=0}^k\binom{k}{j}\alpha^{k-j }( - \epsilon_k\omega )^j \wedge \omega^{n-k}}{\omega^n} = \sum_{j=0}^k \binom{k}{j} \frac{(-\epsilon_k)^j \alpha^{k-j }\wedge \omega^{n-k+j }}{\omega^n} \\
  = \binom{k}{0} \sigma_k(\alpha) - \binom{k}{1}\epsilon_k \sigma_{k-1}(\alpha)+ \cdots + \binom{k}{k-1} (-\epsilon_k )^{k-1} \sigma_1(\alpha)  +(-\epsilon_k)^k.
  \end{aligned}
 \end{equation}
 When $k$ is odd, it follows that 
 \[
  \frac{(\alpha - \epsilon_k\omega )^k \wedge \omega^{n-k}}{\omega^n} =\left(  \sigma_k(\alpha) - k\epsilon_k \sigma_{k-1}(\alpha)\right)+ \cdots + (\epsilon_k )^{k-1} (k \sigma_1(\alpha)  -\epsilon_k).
 \]
 To ensure that $\alpha - \epsilon \omega \in \Gamma_m(\omega)$, it suffices that 
 \[
 \epsilon_k \le \frac{ \binom{k}{0} \sigma_k(\alpha)}{\binom{k}{1}\sigma_{k-1}(\alpha)},\epsilon_k \le \frac{\binom{k}{2} \sigma_{k-2}(\alpha)}{\binom{k}{3} \sigma_{k-3}(\alpha)}, \cdots , \epsilon_k \le k \sigma_1(\alpha).
 \]
 Moreover, Newton's inequality \eqref{ine:Newton} yields that $k \mapsto \frac{\sigma_k(\alpha)}{\sigma_{k-1}(\alpha)}$ is decreasing in $k$. A direct computation implies that $\frac{\binom{k}{j}}{\binom{k}{j+1}} = \frac{j+1}{k-j}$ is increasing in $j$. It thus suffices to choose 
 \[
 \epsilon_k = \frac{1}{k} \frac{\sigma_k(\alpha)}{\sigma_{k-1}(\alpha)}.
 \]
 When $k$ is even, then \eqref{eqn:expandmpositive} can be written as 
 \[
 \begin{aligned}
 \frac{(\alpha - \epsilon_k\omega )^k \wedge \omega^{n-k}}{\omega^n} = \left( \binom{k}{0} \sigma_k(\alpha) - \binom{k}{1}\epsilon_k \sigma_{k-1}(\alpha)\right)+ \cdots \\+  \left((\epsilon_k )^{k-2} \binom{k}{k-2}\sigma_2(\alpha)- \binom{k}{k-1}\sigma_1(\alpha)\epsilon_k   \right)+\epsilon_k^k.
  \end{aligned}
 \]
 Together with Newton's inequality \eqref{ine:Newton}, it therefore suffices to choose 
 \[
 \epsilon_k = \frac{1}{k}\frac{\sigma_k(\alpha)}{\sigma_{k-1}(\alpha)}.
 \]
 For every $1\le k \le m$, there exists $\epsilon_k = \frac{1}{k} \frac{\sigma_k(\alpha)}{\sigma_{k-1 }(\alpha)}$ such that $\alpha - \epsilon_k \omega \in \Gamma_k(\omega)$. To make sure that $\alpha - \epsilon \omega \in \Gamma_k(\omega)$ for all $k$, it is enough to take $\epsilon = \min_k{\epsilon_k}$. It follows from Newton's inequality \eqref{ine:Newton} again that $\frac{1}{k} \frac{\sigma_k(\alpha)}{\sigma_{k-1 }(\alpha)}$ is the smallest. This concludes the argument. 
\end{proof}
As a consequence, one can prove the following property for $m$-positive forms, which we will use hereafter.
\begin{prop} \label{prop:espforms}
    Let $\alpha, \gamma \in \Gamma_m(\omega)$. Suppose that $S_m (\alpha)>0$. Then 
  \[
  \frac{\alpha^{m-1} \wedge \gamma \wedge \omega^{n-m}}{\omega^n} \cdot \frac{\alpha \wedge \omega^{n-1} }{\omega^n} \ge \frac{1}{m!}\frac{\alpha^m \wedge \omega^{n-m} }{\omega^n} \cdot \frac{\gamma \wedge \omega^{n-1} }{\omega^n}.
  \]
\end{prop}
\begin{proof}
    We will prove the inequality
    \[
    \frac{\alpha^{k-1} \wedge \gamma \wedge \omega^{n-k}}{\omega^n} \cdot \frac{\alpha \wedge \omega^{n-1} }{\omega^n} \ge \frac{1}{k!} \frac{\alpha^k \wedge \omega^{n-k} }{\omega^n} \cdot \frac{\gamma \wedge \omega^{n-1} }{\omega^n}
    \]
    for all $k \in \{ 1, \cdots, m \}$, by induction in $k$. When $k=1$, the inequality follows trivially.
    Assume now that, for $k\leq m-1$,
     \[
    \frac{\alpha^{k-1} \wedge \gamma \wedge \omega^{n-k}}{\omega^n} \cdot \frac{\alpha \wedge \omega^{n-1} }{\omega^n} \ge \frac{1}{k!} \frac{\alpha^k \wedge \omega^{n-k} }{\omega^n} \cdot \frac{\gamma \wedge \omega^{n-1} }{\omega^n}.
    \]
    We now show that
    \[
    \frac{\alpha^{k} \wedge \gamma \wedge \omega^{n-k-1}}{\omega^n} \cdot \frac{\alpha \wedge \omega^{n-1} }{\omega^n} \ge \frac{1}{(k+1)!}\frac{\alpha^{k+1} \wedge \omega^{n-k-1} }{\omega^n} \cdot \frac{\gamma \wedge \omega^{n-1} }{\omega^n}.
    \]
    Lemma~\eqref{lem: elemhyppol} implies that there exists $\epsilon_{k+1} = \frac{1}{k+1}\frac{\sigma_{k+1}(a)}{\sigma_{k}(a)}$ such that $\alpha - \epsilon_{k+1} \omega$ is $(k+1)$-positive. Therefore,
    \begin{equation} \label{eqn:induc}
    \begin{aligned}
    \alpha^{k} \wedge \gamma \wedge \omega^{n-k-1} &= \alpha^{k-1}\wedge (\alpha - \epsilon_{k+1} \omega + \epsilon_{k+1} \omega) \wedge \gamma \wedge \omega^{n-k-1} \\
    &\ge \alpha^{k-1} \wedge \epsilon_{k+1} \omega \wedge \gamma \wedge \omega^{n-k-1 }  \\
    &=  \frac{1}{k+1}  \frac{\sigma_{k+1}}{\sigma_{k}} \alpha^{k-1} \wedge \gamma \wedge \omega^{n-k-1}.
    \end{aligned}
    \end{equation}
    By the induction hypothesis, we deduce that
    \[
    \alpha^{k-1} \wedge \gamma \wedge \omega^{n-k} \ge \frac{1}{k!} \frac{\sigma_k}{\sigma_1} \gamma \wedge \omega^{n-1}.
    \]
    Together with \eqref{eqn:induc}, we conclude that
    \[
    \alpha^{k} \wedge \gamma \wedge \omega^{n-k-1} \ge \frac{1}{(k+1)!} \frac{\sigma_{k+1}(\alpha)}{\sigma_1(\alpha)} \gamma \wedge \omega^{n-1},
    \]
    and this completes the proof.
\end{proof} 
\begin{remark}
   Proposition~\eqref{prop:espforms} is a generalization of Lemma 2.2 in \cite{KN15weaksol}, where they assume that $\gamma$ is positive. 
\end{remark}
\begin{remark}
   Proposition~\eqref{prop:espforms} is equivalent to saying that for all $\lambda, \mu \in \Gamma_m$, it follows that
   \[
   \sum_{p} \sigma_{m-1}( \lambda ;p) \mu_{p} \sigma_1(\lambda) \ge \frac{1}{m!} \sigma_m(\lambda ) \sigma_1 (\mu).
   \]
   Therefore,
   \[
   \sum_{p} S_{m-1}( \lambda ;p) \mu_{p} S_1(\lambda) \ge \frac{1}{(m-1)!n} S_m(\lambda ) S_1 (\mu).
   \]
\end{remark}

\subsection{Complex Hessian equations}
\subsubsection{Definition}
We recall the definition of $(\omega,m)$-subharmonic functions following \cite{GuNguyen2018} and briefly review some results related to both $\omega$-plurisubharmonic and $(\omega,m)$-subharmonic functions that will be used later.

Fix an integer $1 \le m \le n$. Fix $\Omega $ an open set in $\C^n$. Given a hermitian metric $\omega$ on $\Omega$, a $\mathcal{C}^2(\Omega)$ function $u : \Omega \rightarrow  \R$ is called 
harmonic with respect to $\omega$ if $\ddc u \wedge \omega^{n-1}=0$ at all points in $\Omega$.

\begin{definition}
A function $u:\Omega \rightarrow \{ -\infty\} \cup \R$ is subharmonic with respect to $\omega$ if:
\begin{enumerate}[label=(\alph*)]
\item $u$ is upper semicontinuous and $u \in L^1_{loc}(\Omega)$;
\item for every relatively compact open set $D \Subset \Omega$ and every function $h\in \mathcal{C}^0(D)$ that is harmonic with respect to $\omega$ on $D$ the following implication holds:
\[
u \le h \text{ on }\partial D \implies u \le h \text{ in }D.
\]
\end{enumerate}
\end{definition}

\begin{definition}
A function $\varphi : \Omega \rightarrow  \{ -\infty\} \cup \R$ is quasi-subharmonic with respect to $\omega$ if locally $\varphi = u + \rho$, where $u$ is subharmonic with respect to $\omega $ and $\rho$ is smooth.

A function $\varphi$ is $\alpha$-subharmonic with respect to $\omega$ if $\varphi$ is quasi-subharmonic with respect to $\omega$ and $(\alpha + \ddc \varphi) \wedge \omega^{n-1}\ge 0$ in the sense of distributions.
\end{definition}
The positive cone $\Gamma_m(\omega)$ associated with the metric $\omega$ is defined by 
\[
\left\{  \gamma \text{ real }(1,1) \text{-form}: \gamma^k\wedge \omega^{n-k}>0, k=1,\cdots,m \right\}.
\]
It follows from Gårding's inequality\cite{gar59} that if $\gamma_0, \gamma_1, \cdots, \gamma_{m-1}\in \Gamma_{m}(\omega)$, then 
\[
\gamma_0 \wedge\gamma_1\wedge \cdots \wedge\gamma_{m-1} \wedge\omega^{n-m} >0. 
\]
One can write $\widetilde{\omega}= \gamma_1\wedge \cdots \wedge\gamma_{m-1} \wedge\omega$, and it is a strictly positive $(n-1,n-1)$-form on $\Omega$.
\begin{definition}
A function $\varphi: \Omega \rightarrow [ -\infty, + \infty)$ is called $(\alpha,m)$-subharmonic with respect to $\omega$ if $\varphi$ is $\alpha$-subharmonic with respect to $\widetilde{\omega}$ in $\Omega$ for all 
$\widetilde{\omega}$ of the form $\widetilde{\omega}^{n-1}= \gamma_1\wedge \cdots \wedge\gamma_{m-1} \wedge\omega$, where $\gamma_1, \cdots,\gamma_{m-1} \in \Gamma_m(\omega)$.

A function $u : X \rightarrow \mathbb{R} \cup \{- \infty \} $ is called $(\alpha,m)$-subharmonic on $X$ if $u$ is $(\alpha,m)$-subharmonic on each local chart $U$ of $X$.
\end{definition}
The set of all locally integrable functions on $U$ which are $(\alpha,m)$-subharmonic with respect to $\omega$ in $U$ is denoted by $\SH_{\omega,m}(U,\alpha)$. In many contexts, such as \cite{guedj2023degeneratecomplexhessianequations}, the forms $\omega$ and $\alpha$ may differ.

However, in this paper, we focus on the set $\SH_{\omega,m}(U,\omega)$. To simplify the notations, we denote by $\SH_m(U,\omega)$ the set of all $(\omega,m)$-subharmonic functions with respect to $\omega$ in $U$. The set of all $\omega$-plurisubharmonic functions on $U$ is denoted by $\PSH(U, \omega)=\SH_n(U,\omega)$.

For notational convenience, we denote $\omega_{u} \coloneqq \omega + \ddc u$.

\begin{remark}
  By Gårding's inequality\cite{gar59}, if $u \in \mathcal{C}^2(X)$, then $u$ is $(\omega,m)$-subharmonic with respect to $\omega$ on $X$ if and only if the associated form $\omega _u$ belongs to the closure of $\Gamma_m(\omega) $.
\end{remark}
\subsubsection{Domination principle}
\begin{prop}[Domination principle] \label{prop: DP}
Assume that $u, v \in \SH_m(X, \omega) \cap L^{\infty}(X)$ satisfy
\[
1_{\{u<v\}}\left(\omega+d d^c u\right)^m \wedge \omega^{n-m} \leq c 1_{\{u<v\}}\left(\omega+d d^c v\right)^m \wedge \omega^{n-m},
\]
for some constant $0 \le c<1$. Then $u \ge v$.
\end{prop}
We refer the readers to \cite[Proposition 1.8]{fang2025continuitysolutionscomplexhessian} for the proof. The proof applied the comparison principle \cite[Corollary 6.2]{kolodziej2023complexhessianmeasuresrespect} of Hessian operators associated with bounded $(\omega,m)$-subharmonic functions. The case in which $u,v$ are continuous has been treated in \cite[Section 2]{guedj2023degeneratecomplexhessianequations}. As an direct application, we have the following Corollary.
\begin{coro}
     Let $u, v \in \SH_m(X, \omega) \cap L^{\infty}(X)$. If $e^{- v} \omega_v^m \wedge \omega^{n-m} \ge e^{- u} \omega_u^m \wedge \omega^{n-m}$, then $u \ge v$.
\end{coro}
\begin{proof}
    On the set $\{u <v \}$, we have
    \[
    (\omega+ \ddc u)^m \wedge \omega^{n-m} \le e^{u-v } (\omega + \ddc v)^m \wedge \omega^{n-m } \le c (\omega + \ddc v)^m \wedge \omega^{n-m},
    \]
    where $0<c<1$. We deduce from the Domination principle (Proposition~\ref{prop: DP}) that $u\ge v$.
\end{proof}
Applying the last corollary twice, we obtain
\begin{coro} \label{coro:dp2}
     Let $u, v \in \SH_m(X, \omega) \cap L^{\infty}(X)$. If $e^{- v} \omega_v^m \wedge \omega^{n-m} = e^{- u} \omega_u^m \wedge \omega^{n-m}$, then $u = v$.
\end{coro}

\section{Estimates} \label{sec:est}
In this section, we establish an a priori $C^2$-estimate for smooth solutions to the Hessian equation 
\[
(\omega + \ddc u)^m \wedge \omega^{n-m} = e^{u + h (x,v)} \omega^n.
\]
The $C^2$-estimate is independent of $\Delta v$ and $|\nabla v |^2$. Let us precise our second order estimate:
\begin{thm} \label{thm:c2estmate}
Assume that $v \in C^{\infty}(X) \cap \SH_m(X, \omega) $, with $-A_0 \le v \le A_0$, and $h: X \times [-A_0 , A_0 ]\rightarrow \R $ is smooth.  Assume that $u \in C^{\infty}(X) \cap \SH_m(X, \omega)$ is a solution to the equation
\[
(\omega + \ddc u)^m \wedge \omega^{n-m} = e^{u + h (x,v)} \omega^n.
\]
Then, for some fixed $\varepsilon_0>0$, we have 
\[
\sup_X \Delta_{\omega} (u) \le C(1 + \sup_X |\nabla u|^2 ),
\]
where $C$ is a constant depending on $A_0, B_0, X ,n, \omega, \tilde{\varphi}$, and $C^2$-norm of $h(x,t)$. Here $B_0$ is an upper bound for the $C^2$-norm of some smooth function $\tilde{\varphi}$ such that $|\tilde{\varphi}-v|<\epsilon_0$.
\end{thm}
In the proof of the regularity result, we approximate a given bounded (hence continuous) solution $\varphi$ by smooth $(\omega,m)$-subharmonic functions $v_j$. We will also fix a smooth function $\tilde{\varphi}$ such that $|\tilde{\varphi}-\varphi|<\epsilon_0$. Since $\varphi$ is continuous we also have $|\tilde{\varphi}-v_j|<2\epsilon_0$ independent of $j$. In this way, our $C^2$ estimate does not depend on $j$. 

We follow the proof of Proposition 13 in the work of Sz\'ekelyhidi \cite{Sze18}, which generalized the $\mathcal{C}^2$-estimate of Hou--Ma--Wu \cite{HMW10}. We need some modifications since our equation now is different.
Let us stress that our $\mathcal{C}^2$-estimate should not depend on $|\nabla v |^2$ and $\Delta v$. This is the main difficulty in our argument and the principal difference from the proof in Sz\'ekelyhidi's work \cite{Sze18}. In the proof below, we only present the necessary modifications for our setting and highlight some of the key steps.
\begin{proof}
The uniform estimate for solution $u$ is a direct result from the maximum principle (see \cite[Theorem 3.17]{Ko_odziej_2016}).
To simplify the notation, we use the following convention. We denote
\[
C_1 = \sup_X|u|, 
\]
and we fix $C_2$ such that 
\[
C_2 \ge  \| h\|_{C^2( X \times [-A_0, A_0])}, \; C_2 \ge  \| h\|_{C^1( X \times [-A_0, A_0])} , C_2 \ge \sup_{x\in X, t \in [-A_0, A_0]} \| h(x,t)\|,
\]
and we fix $C_3$ and $C_4$ such that
\[
 C_3\le\frac{1}{m}e^{\frac{u + h (x,v)}{m} }  \le C_4.
\]
We denote by $C_0$, $C_0'$ uniform constants that might be different each line.
Fix $\epsilon >0$ to be a small constant that we will choose later. We take a smooth function $\tilde{\varphi}$, such that
\[
|v - \tilde{\varphi } | \le \epsilon.
\]
We aim to apply the maximum principle to the function
\[
G(z) =  \log (\lambda_1) + \Phi(| \nabla u |^2 ) + \Psi(u) + \chi(v-\tilde{\varphi}).
\]
The function $\Phi$ and $\Psi$ are same as the functions used in  \cite{Sze18}:
\[
\Phi(t) = - \frac{1}{2} \log\left(1 -\frac{t}{2K} \right), \; \Psi(t ) = - 2At + \frac{A \tau}{2} t^2,
\]
where $A$ and $\tau$ are uniform constant that we will precise later, and $K = \sup_X|\nabla u|^2 +1$. The function $\chi$ is defined as 
\[
\chi(t ) = - \frac{1}{12}\log \left(  2\epsilon  -t  \right).
\]
Therefore, $\frac{1}{36\epsilon} \le \chi' \le \frac{1}{12\epsilon}$, and $\chi''= 12 (\chi' )^2$.
As in \cite[Session 4]{Sze18}, we slightly modify $W^{ij}=g^{i \bar{k}} \tilde{g}_{k \bar{j} }$.
We consider the quantity 
\begin{equation}
    \widetilde{G}(z) = \log (\widetilde{\lambda_1}) + \Phi(| \nabla u |^2 ) + \Psi(u) + \chi(v -\tilde{\varphi}),
\end{equation}
where $\widetilde{\lambda_1}$ denotes the largest eigenvalue of $\widetilde{W}= W-B$, with $W^{ij}=g^{i \bar{k}} \tilde{g}_{k \bar{j} } $ and $B$ defined as in \cite[(69)]{Sze18}. 
Assume that $G$ attains its maximum at $x_0$. We claim that $\tilde{G}$ attains its maximum at $x_0$ in small neighborhood of $x_0$. Since $\widetilde{W}= W-B$ and $B$ is a diagonal matrix such that $0 = B^{11} <B^{22} < \cdots <B^{nn}$, we get that that $\widetilde{G}(z) \le G(z)$ in the neighborhood of $x_0$ and thus, $\widetilde{\lambda_1}(W(z)-B) \le \lambda_1(W(z))$. It follows that $\widetilde{G}(z) \le G(z)\le G(x_0) = \widetilde{G}(x_0)$. We note that the difference between the quantity we consider and that in \cite{Sze18} is that we have added an extra function of $v-\tilde{\varphi}$. This term is introduced to compensate for the $| v_p |^2$ and $v_{p \bar{p} }$ terms arising when we differentiate our Hessian equation.

We then have
\[
\begin{aligned}
    G_p =& \frac{\tilde{g}_{1\bar{1}p } }{\lambda_{1 }} + \Phi'(  u_{kp} u_{\bar{k} } + u_{k} u_{\bar{k} p} ) + \Psi ' u_p +w_p +\chi' v_p - \chi' \tilde{\varphi}_p ,\\
    G_{p \bar{p}} =& (\log (\lambda_1))_{p \bar{p} } 
+\Phi^{\prime \prime}\left| u_{k p} u_{\bar{k}}+ u_p u_{\bar{p} p}\right| ^2 
+\Phi^{\prime}\left(  |u_{p j}|^2  + |u_{\bar{p} p}|^2 + 2 \Re (u_{k \bar{p} \bar{p}} u_{\bar{k}})\right)\\
+&  \Psi^{\prime \prime} |u_p|^2 +\Psi^{\prime} u_{p \bar{p}} +\chi''|v_p -\tilde{\varphi}_p|^2 + \chi'(v_{p \bar{p} } - \tilde{\varphi}_{p\bar{p} }),
\end{aligned}
\]
where $w_p = \frac{-(B^{11})_p}{\lambda_1}$. We briefly recall the notations in \cite{Sze18} for the convenience of the reader. We denote 
\[
F^{pq } = \frac{\partial F }{ \partial w^{pq} }, F^{pq ,rs} = \frac{\partial^2 F }{\partial w^{pq}\partial w^{rs} },
\]
where $w^{pq}$ is the $pq$-entry of the matrix $W=g^{p \bar{k} } \tilde{g}_{k \bar{q} } $. We write $F(W) = S_m(\lambda_1, \cdots, \lambda_n )$. It then follows from direct computations that
\begin{equation}
    F^{pp}= \frac{1}{m} S_m(\lambda)^{\frac{1}{m} -1 } S_{m-1 } (\lambda ;p),
\end{equation}
and 
\begin{equation}
\begin{aligned}
    F^{pq,rs } =\left( \frac{1}{m} S_m^{\frac{1}{m} -1} S_m(\lambda;pr) + \frac{1}{m} \left(\frac{1}{m} -1 \right) S_{ m}^{\frac{1}{m}-2 }(\lambda) S_{m-1}(\lambda;p) S_{m-1}(\lambda;r) \right) \delta_{pq} \delta_{rs} \\
    + \frac{F^{pp} -F^{qq}}{\lambda_p -\lambda_q } (1-\delta_{pq}) \delta_{ps} \delta_{qr}.
\end{aligned}
\end{equation}
We denote by $\mathcal{F}$ the sum of $F^{pp}$ in the index $p$. We note that $\mathcal{F}$ has lower bound:
\begin{equation}
\begin{aligned}
\mathcal{F} = \frac{1}{m} S_m(\lambda)^{\frac{1}{m} -1 } \sum_p S_{m-1} (\lambda;p) = \frac{n-m+1}{m} S_m(\lambda)^{\frac{1}{m} -1 } S_{m-1}(\lambda)\\ \ge  \frac{n-m+1}{m} S_m(\lambda)^{\frac{1}{m} -1 } S_{m}(\lambda)^{\frac{m-1}{m}} = \frac{n-m+1}{m}.
\end{aligned}
\end{equation}
We differentiate the equation
\[
F\left( g^{p \bar{k} } \tilde{g}_{k \bar{q} }  \right) = e^{\frac{u + h(x,v)}{m} }.
\]
We then obtain
\begin{equation} \label{frml:Ffirst}
   F^{pp } \tilde{g}_{p \bar{p} l }   =  \frac{1}{m}e^{\frac{u + h(x,v)}{m} } (u_{l} + h_l(x,v) + h_v(x,v) v_l ) , 
\end{equation}
and 
\begin{equation}
\begin{aligned} \label{frml:Fsecond}
F^{p q ,r s } \tilde{g}_{p \bar{q} l } \tilde{g}_{r \bar{s} \bar{l} }  + F^{ pp }\tilde{g}_{p \bar{p } l \bar{l} }  
=\frac{1}{m^2} e^{\frac{u + h(x,v)}{m} } |u_l  + h_l(x,v) + h_v(x,v)v_l |^2 \\
+ \frac{1}{m}e^{\frac{u + h(x,v)}{m} } (u_{l \bar{l} } + h_{l \bar{l} }(x,v)+ h_{lv }(x,v)v_{\bar{l}}+ h_{v\bar{l} }(x,v)v_l  +h_{vv}(x,v)|v_l |^2 + h_v(x,v)v_{l \bar{l} } ).
\end{aligned}
\end{equation}
It follows that
\[
\begin{aligned}
F^{pp}\tilde{g}_{p \bar{p } 1 \bar{1} }  \ge C_2\lambda_1+ C_4(-1 -C_2 -3C_2|v_1|^2-C_2 |v_{1\bar{1}}| ) - F^{p q ,r s }  \tilde{g}_{p \bar{q} \bar{1}} \tilde{g}_{r \bar{s} 1 }  ,
\end{aligned}
\]
and we can absorb the constant $C_4$ in $C_0 \mathcal{F}$.
Combining this with (78) and (83) in \cite{Sze18}, we obtain the estimate for $L(\log{\tilde{\lambda}_1 } )$:
\[
\begin{aligned}
L(\log{\tilde{\lambda}_1 } ) \ge \frac{C_4}{\lambda_1}(-3C_2|v_1|^2-C_2 |1+v_{1\bar{1}}| )\\
- \frac{F^{pq,rs}\tilde{g}_{p \bar{q}1 } \tilde{g}_{r\bar{s} \bar{1} } }{\lambda_1}- \frac{F^{pp }|\tilde{g}_{p\bar{1}1 } |^2}{\lambda_1^2}- C_0(\mathcal{F} + \lambda_1^{-2}|F^{pp} \tilde{g}_{1\bar{1}k }| ).
\end{aligned}
\]
Using \eqref{frml:Ffirst} and commuting derivatives as in \cite{Sze18}, we obtain
\[
\begin{aligned}
\Phi' F^{pp } u_{k p \bar{p} } u_{\bar{k} } = \Phi' F^{pp } (u_{p \bar{p}k } - T^{ p}_{pk} \lambda_p + T^q_{pk} g_{q \bar{p} } + R^q_{p\bar{p}k } u_q ) u_{\bar{k} } \\
= \Phi' \frac{1}{m}e^{\frac{u + h(x,v)}{m} } (u_{k} + h_k(x,v) + h_v(x,v) v_k ) u_{\bar{k} } - \Phi' F^{pp } g_{p \bar{p} k } u_{\bar{k} } \\- \Phi'F^{pp}T^{ p}_{pk} \lambda_p u_{\bar{k} } +\Phi'F^{pp} T^p_{pk}u_{\bar{k} }  +\Phi'F^{pp} R^q_{p\bar{p}k } u_q u_{\bar{k} },
\end{aligned}
\]
where the torsion terms can be estimated as (97), (98) in \cite{Sze18}, and so does its conjugate. We thus get
\[
\begin{aligned}
  &2\Phi' \sum_k \Re (F^{pp } u_{k p \bar{p} } u_{\bar{k} } ) \\
  = &\Phi^{\prime} 2\Re\sum_{k=1}^n\left(\frac{1}{m}e^{\frac{u + h(x,v)}{m} } (u_k  + h_k(x,v)+ h_v(x,v)v_k )\right) u_{\bar{k}} -C_0 \mathcal{F} - \epsilon_1 K^{-1}F^{pp } \lambda_p^2 - C_{\epsilon_1 } \mathcal{F}  \\
  =& \Phi' \frac{2}{m} e^{\frac{u + h(x,v)}{m} } |\nabla u|^2 + \Phi'2 \sum_k\Re \left(\frac{1}{m} e^{\frac{u + h(x,v)}{m} } (h_k(x,v) + h_v(x,v)v_k) u_{\bar{k} } \right)  -C_0 \mathcal{F} - \epsilon_1 K^{-1}F^{pp } \lambda_p^2 - C_{\epsilon_1 } \mathcal{F} \\
 \ge& -\frac{1}{K} \sum_k \left|\frac{1}{m} e^{\frac{u + h(x,v)}{m} } h_k(x,v)\right|K^{1/2 }- \frac{1}{K} \sum_k \left|\frac{1}{m} e^{\frac{u + h(x,v)}{m} } h_v(x,v) v_k u_{\bar{k} } \right| -C_0 \mathcal{F} - \epsilon_1 K^{-1}F^{pp } \lambda_p^2 - C_{\epsilon_1 } \mathcal{F}\\
 \ge & - \frac{1}{2K} \sum_k \frac{A_1 K|v_k|^2}{\Tr_{\omega}(\omega_{u})}-  \sum_k \frac{1}{2K^2A_1} \left| C_2 C_4 u_{\bar{k}} \right|^2 \Tr_{\omega}(\omega_{u}) -C_0 \mathcal{F} - \epsilon_1 K^{-1}F^{pp } \lambda_p^2 - C_{\epsilon_1 } \mathcal{F},
\end{aligned}
\]
where $A_1$ is a positive uniform constant to be precised later. We remark that the last inequality follows from Cauchy--Schwarz inequality and we absorbed the $-\frac{1}{K} \left|\frac{1}{m} e^{\frac{u + h(x,v)}{m} } h_k(x,v)\right|K^{1/2 }$ term into $C_0 \mathcal{F}$. 
We use half of $F^{pp}\lambda_p^2$ term to absorb the $|u_{p\bar{p} } |^2$ and $|u_{kp}|^2$ term as (100) in \cite{Sze18}. It follows that
\begin{equation}
\begin{aligned}
    0 \ge& F^{pp} G_{p\bar{p} } \ge \frac{C_4}{\lambda_1}( -3C_2|v_1|^2-C_2 |1+v_{1\bar{1}}| )- \frac{F^{pq,rs}\tilde{g}_{p \bar{q}1 } \tilde{g}_{r\bar{s} \bar{1} } }{\lambda_1}- \frac{F^{pp }|\tilde{g}_{p\bar{1}1 } |^2}{\lambda_1^2} \\
+&\Phi^{\prime \prime}\sum_k F^{pp}\left| u_{k p} u_{\bar{p}}+ u_p u_{\bar{p} p}\right| ^2 + \frac{1}{32K}F^{pp} \lambda_p^2 +\frac{1}{16K} F^{pp}\sum_k\left(  |u_{k p}|^2  + |u_{\bar{p} p}|^2 \right)\\
+&  \Psi^{\prime \prime} F^{pp}|u_p|^2 +\Psi^{\prime}F^{pp} u_{p \bar{p}} +\chi''F^{pp}|v_p - \tilde{\varphi}_p|^2 + \chi'F^{pp}(v_{p \bar{p} }- \tilde{\varphi}_{p\bar{p} } ) \\
 - &\frac{1}{2} \sum_k  \frac{A_1 |v_k|^2}{\Tr_{\omega}(\omega_{u})}-  \sum_k \frac{1}{2K^2A_1} \left| C_2 C_4 u_{\bar{k}} \right|^2 \Tr_{\omega}(\omega_{u}) -C_0 \mathcal{F} - C_0 \lambda_1^{-2}|F^{pp} \tilde{g}_{1\bar{1}k }| .  
\end{aligned}
\end{equation}
The identity $G_p =0$ implies that
\begin{equation} \label{eqn:Gp=0}
\frac{\tilde{g}_{1\bar{1} p} }{\lambda_1} = - \Phi'(  u_{kp} u_{\bar{k} } +  u_{p} u_{\bar{p} p} ) - \Psi ' u_p - w_p- \chi' (v_p- \tilde{\varphi}_p) .
\end{equation}
It follows that
\[
\begin{aligned}
\frac{ F^{pp}|\tilde{g}_{1\bar{1} p} |}{\lambda_1^2} \le \frac{1}{2K\lambda_1} F^{pp} (|u_{kp}| + |u_{k\bar{p} }| )K^{1/2} + C_0 \mathcal{F} + \frac{2A}{\lambda_1} F^{pp } |u_p | + \chi'\lambda_1^{-1}|v_p-\tilde{\varphi}_p |F^{p p } \\
\le \frac{1}{2K\lambda_1} F^{pp} (|u_{kp}| + |u_{k\bar{p} }| )K^{1/2} + C_0 \mathcal{F} + \frac{2A}{\lambda_1} F^{pp } |u_p | + \chi'\lambda_1^{-1}F^{pp }|v_p| + \chi' B_0\lambda_1^{-1}\mathcal{F}.
\end{aligned}
\]
As in \cite{Sze18}, we use the $|u_{kp }|^2 + |u_{p\bar{p} } |^2$ terms to absorb the negative $|u_{kp}| + |u_{k\bar{p} }|$ terms above, and we finally get that
\begin{equation} \label{eqn:FppGpp2}
    \begin{aligned}
    0 \ge& F^{pp} G_{p\bar{p} } \ge \frac{C_4}{\lambda_1}( -3C_2|v_1|^2-C_2 |1+v_{1\bar{1}} |)- \frac{F^{pq,rs}\tilde{g}_{p \bar{q}1 } \tilde{g}_{r\bar{s} \bar{1} } }{\lambda_1}- \frac{F^{pp }|\tilde{g}_{p\bar{1}1 } |^2}{\lambda_1^2} \\
+&\Phi^{\prime \prime}F^{pp}\left| u_{k p} u_{\bar{k}}+ u_p u_{\bar{p} p}\right| ^2 + \frac{1}{32K}F^{pp} (\lambda_p^2 + |u_{pk } |^2) \\
+& \Psi^{\prime \prime} F^{pp}|u_p|^2 +\Psi^{\prime}F^{pp} u_{p \bar{p}} +\chi''F^{pp}|v_p -\tilde{\varphi}_p |^2 + \chi'F^{pp}(v_{p \bar{p} } - \tilde{\varphi}_{p\bar{p} } ) \\
 - &\frac{1}{2} \sum_k A_1 \frac{|v_k|^2}{\Tr_{\omega}(\omega_{u})}-  \sum_k \frac{1}{2A_1 K^2} \left| C_2 C_4 u_{\bar{k} } \right|^2 \Tr_{\omega}(\omega_{u})\\
 -& C_0\left( \mathcal{F} +\chi' \lambda_1^{-1 } \mathcal{F}  +\frac{2A}{\lambda_1} F^{p\bar{p} } |u_p | + \frac{\chi'}{\lambda_1}F^{p p }|v_p |\right).  
\end{aligned}
\end{equation}
We deal with two cases separately.

\textbf{Case 1:}  $\lambda_n < -\delta \lambda_1$.
To deal with the third term in \eqref{eqn:FppGpp2}, we use the identity $G_p =0$:
\[
\begin{aligned}
&-\frac{F^{p\bar{p} } |\tilde{g}_{1\bar{1} p  }|^2 }{(\lambda_1)^2} = -F^{p\bar{p} }\left|\Phi^{\prime}\left(\sum_k u_{k p} u_{\bar{k}}+u_p u_{\bar{p} p}\right) + \Psi^{\prime} u_p + \chi'(v_p-\tilde{\varphi}_p )+w_p \right|^2 \\
 \ge& -2 (\Phi' )^2 F^{p\bar{p} } \left| \sum_k u_{k p} u_{\bar{k}}+u_p u_{\bar{p} p} \right|^2 -6 (2A)^2 F^{p\bar{p} } |u_p|^2 - 6 (\chi')^2 F^{p\bar{p} } |v_p-\tilde{\varphi}_p |^2 -C_0\mathcal{F} \\
 \ge& -2 (\Phi' )^2 F^{p\bar{p} } \left| \sum_k u_{k p} u_{\bar{k}}+u_p u_{\bar{p} p} \right|^2 -6 (2A)^2 F^{p\bar{p} } |u_p|^2 - 6 (\chi')^2 F^{p\bar{p} } \left(\frac{3}{2}|v_p|^2  + 3|\tilde{\varphi}_p |^2\right) -C_0\mathcal{F} .
\end{aligned}
\]
Commutating the derivatives as (106) in \cite{Sze18} and using \eqref{eqn:Gp=0} to deal with the $|\tilde{g}_{1\bar{1}p } |$ term, we get that
\[
\begin{aligned}
&\frac{F^{pp }|\tilde{g}_{p\bar{1} 1 } |^2}{\lambda_1^2}\le \frac{F^{pp }|\tilde{g}_{1\bar{1} p } |^2}{\lambda_1^2} + C_0F^{pp }(1 + \lambda_1^{-1} |\tilde{g}_{1\bar{1} p } |) \\
\le& \frac{F^{pp }|\tilde{g}_{1\bar{1} p } |^2}{\lambda_1^2}+ C_0 \mathcal{F} + \frac{C_0}{2 K}F^{pp} (|u_{kp}| + |u_{p \bar{p}}| )K^{1/2} + 2A C_0 F^{pp} |u_p| + \chi' C_0 |v_p - \tilde{\varphi}_p| F^{pp}.
\end{aligned}
\]
We also have 
\[
2A C_0 F^{pp} |u_p| \le A^2 F^{pp} |u_p|^2 + C_0^2 \mathcal{F}
\]
and
\[
\chi' C_0 |v_p -\tilde{\varphi}_p | F^{pp} \le F ^{pp} \chi'|v_p -\tilde{\varphi}_p|^2 + C_0^2 \chi' \mathcal{F} \le F^{pp} 2\chi' |v_p |^2 + (C_0^2+ 2 B_0^2 )\chi' \mathcal{F}  .
\]
To deal with the last term in \eqref{eqn:FppGpp2}, we assume that $\lambda_1 \ge C_0$. Hence, 
\[
-C_0\left( \frac{2A}{\lambda_1} F^{p\bar{p} } |u_p | + \frac{\chi'}{\lambda_1}F^{p p }|v_p |\right) \ge -2A F^{p\bar{p} } |u_p | - \chi' F^{p p }|v_p | \ge -A^2 \mathcal{F}K -\mathcal{F}- \chi' F^{p p }|v_p |.
\]
We therefore obtain
\begin{equation} \label{eqn:FppGpp3}
    \begin{aligned}
    0 \ge& F^{pp} G_{p\bar{p} } \ge \frac{C_4}{\lambda_1}(-3C_2|v_1|^2-C_2 | 1+v_{1\bar{1}} |)- \frac{F^{pq,rs}\tilde{g}_{p \bar{q}1 } \tilde{g}_{r\bar{s} \bar{1} } }{\lambda_1} \\
+&  \frac{1}{32K}F^{pp} \lambda_p^2 -26A^2 \mathcal{F}K  + \Psi^{\prime \prime} F^{pp}|u_p|^2 +\Psi^{\prime}F^{pp} u_{p \bar{p}}\\
+&(\chi''-  9(\chi')^2 - 2 \chi')F^{pp}|v_p|^2 + \chi'F^{pp}v_{p \bar{p} } - \chi'F^{pp}|v_p| \\
 - &\frac{1}{2} \sum_k A_1 \frac{|v_k|^2}{\Tr_{\omega}(\omega_{u})}-  \sum_k \frac{1}{2A_1K^2} \left|C_4 C_2 u_{\bar{k}} \right|^2 \Tr_{\omega}(\omega_{u}) -C_0(\mathcal{F} + \chi'\mathcal{F} ).  
\end{aligned}
\end{equation}
We now turn to deal with the $|v_p|^2$ and $v_{p \bar{p} }$ terms. We first note that $v $ is an $(\omega,m )$-subharmonic function. We deduce that
\begin{equation} \label{eqn: negativev_pp}
\begin{aligned}
-\frac{C_2 C_4 }{\lambda_1}|1+ v_{1 \bar{1} }| \ge - \frac{n}{S_1 } C_2C_4 |1+v_{1 \bar{1} }|
\ge - \frac{n(n-1)}{S_1 }  C_2C_4 \left(n+\sum_p v_{p \bar{p} }\right) ,
\end{aligned}
\end{equation}
where we used the fundamental inequality $n \lambda_1 \ge S_1 $ and $ |1 + v_{1 \bar{1} }| \le (n-1) \sum_{p}(1 + v_{p\bar{p} })$ (see \eqref{frml:valerabsolupse}).
We deduce from Proposition~\eqref{prop:espforms} that
\[
\begin{aligned}
\chi' F^{p\bar{p} } v_{p\bar{p} } = \chi' F^{p\bar{p} } (1+ v_{p\bar{p} }) - \chi'\mathcal{F} =& \chi' \frac{1}{m}\left[S_m(\lambda)\right]^{\frac{1}{m}-1} S_{m-1}(\lambda ; p)(1+ v_{p \bar{p} }) - \chi'\mathcal{F} \\
\ge &\frac{1}{36 \epsilon} \frac{1}{m}\left[S_m(\lambda)\right]^{\frac{1}{m}-1} \frac{ S_m(\lambda )}{n(m-1)!S_1(\lambda)} S_1(v_{p\bar{p}  }+1 )  - \chi'\mathcal{F} \\
\ge &\frac{1}{36 \epsilon}\frac{C_3}{n(m-1)!S_1(\lambda)} S_1(v_{p\bar{p}  }+1 )  - \chi'\mathcal{F} .
\end{aligned}
\]
We can thus choose $\epsilon$ so that
\[
\epsilon \le \frac{C_3 }{ 36 (m-1)! n^2 (n-1) C_2 C_4},
\]
and so that the negative $\sum_p{v_{p\bar{p} }}$ term in \eqref{eqn: negativev_pp} vanishes. Since we have specified $\epsilon$, the term $\chi'\mathcal{F}$ can now be absorbed into $C_0 \mathcal{F}$. Our inequality of $F^{pp }G_{p\bar{p}}$ becomes
\begin{equation} \label{eqn:FppGpp4}
\begin{aligned}
    0 \ge& F^{pp} G_{p\bar{p} } \ge \frac{C_4}{\lambda_1}(-3C_2|v_1|^2)- \frac{F^{pq,rs}\tilde{g}_{p \bar{q}1 } \tilde{g}_{r\bar{s} \bar{1} } }{\lambda_1} \\
+&  \frac{1}{32K}F^{pp} \lambda_p^2 -26 A^2 \mathcal{F}K  + \Psi^{\prime \prime} F^{pp}|u_p|^2 +\Psi^{\prime}F^{pp} u_{p \bar{p}}\\
+&(\chi''-  9(\chi')^2 -2\chi')F^{pp}|v_p|^2 -\chi'|v_p|F^{p p } \\
 - &\frac{1}{2} \sum_k A_1 \frac{|v_k|^2}{\Tr_{\omega}(\omega_{u})}-  \sum_k \frac{1}{2A_1K^2} \left|C_4 C_2 u_{\bar{k}} \right|^2 \Tr_{\omega}(\omega_{u}) -C_0\mathcal{F}.  
\end{aligned}
\end{equation}
On the other hand, we have 
\[
-3C_4C_2\frac{|v_1|^2}{\lambda_1} \ge -\frac{n}{S_1} 3 C_4C_2|v_1 |^2  \ge -\frac{n}{S_1}   3C_4C_2\sum_p|v_p |^2 .
\] 
We also have, by Cauchy--Schwarz inequality, that
\[
\chi'|v_p|F^{p p } \le \chi' \left(\frac{1}{2 } F^{p\bar{p} }|v_p|^2 + \frac{1}{2 } \mathcal{F}  \right).
\]
It follows from Proposition~\eqref{prop:espforms} that 
\[
\frac{S_1(\lambda)}{S_m(\lambda)} \left(\sum_p |v_p|^2 S_{m-1}(\lambda|p) \right) \ge \frac{1}{n(m-1)!} \sum_p |v_p|^2  .
\]
It then implies that
\[
\begin{aligned}
&(\chi''-  9(\chi')^2 -2\chi')F^{pp}|v_p|^2 -\chi'|v_p|F^{p p } \\
\ge &\left(\chi''-\frac{19}{2} \chi'^2 -2\chi' \right)F^{p p } |v_p|^2 -C_0 \chi'\mathcal{F} \\
\ge& \left(\chi''-\frac{19}{2}\chi'^2-2\chi' \right) \frac{1}{m} [S_m(\lambda)]^{\frac{1}{m}-1 }  \frac{1}{n(m-1)!} \frac{S_m(\lambda)}{S_1(\lambda) } \sum_p|v_p|^2 -C_0 \mathcal{F} \\
\ge & \left(\chi''-\frac{19}{2}\chi'^2-2\chi' \right) \frac{1}{n(m-1)!} \frac{C_3}{S_1(\lambda) } \sum_p|v_p|^2 -C_0 \mathcal{F}.
\end{aligned}
\]
Hence, the inequality for $L(G)$ becomes
\begin{equation} \label{eqn:FppGpp4b}
\begin{aligned}
    0 \ge& F^{pp} G_{p\bar{p} } \ge \left( \left(\chi''-  \frac{19}{2}(\chi')^2 -2\chi' \right) \frac{C_3}{n(m-1)!S_1} -\frac{n}{S_1}   3C_4C_2 -\frac{1}{2}  \frac{A_1}{S_1}\right) \sum_p|v_p|^2 \\-&\frac{F^{pq,rs}\tilde{g}_{p \bar{q}1 } \tilde{g}_{r\bar{s} \bar{1} } }{\lambda_1} 
+ \frac{1}{32K}F^{pp} \lambda_p^2 -26A^2 \mathcal{F}K  \\
+& \Psi^{\prime \prime} F^{pp}|u_p|^2 +\Psi^{\prime}F^{pp} u_{p \bar{p}}  
 -  \sum_k \frac{1}{2A_1K^2} \left|C_4 C_2 u_{\bar{k}}\right|^2 \Tr_{\omega}(\omega_{u}) -C_0\mathcal{F}.  
\end{aligned}
\end{equation}
By construction, we have $\chi''= 12(\chi')^2$, it thus suffices to ensure that $2\chi'^2 - 2\chi' \ge \frac{ (m-1)!3 n^2 C_4 C_2}{C_3} + \frac{(m-1)!nA_1}{ 2 C_3} $.
It follows that
\[
\begin{aligned}
\chi''- 10\chi'^2 - 2\chi'-\frac{ m!3 n^2 C_4 C_2}{C_3} -\frac{m! nA_1}{ 2 C_3} \ge 0.
\end{aligned}
\]
Thus, we absorb all terms involving $ |v_p|^2$, while keeping the positive $\frac{1}{2}\chi'^2|v_p|^2$ term for later use in Case 2. We now precise that $\epsilon$ is a uniform small constant so that 
\[
2\left( \frac{1}{36\epsilon} \right)^2 - 2 \frac{1}{12\epsilon} \ge (m-1)!3 n^2 C_4 C_2 + \frac{(m-1)!nA_1}{ 2 C_3} \text{, and } \epsilon \le \frac{C_3 }{ 36 (m-1)! n^2 (n-1) C_2 C_4},
\]
where $A_1$ is a large uniform constant that we will choose in case 2. Hence, the inequality \eqref{eqn:FppGpp4} becomes 
\begin{equation} \label{eqn:FppGpp5}
\begin{aligned}
    0 \ge& F^{pp} G_{p\bar{p} } \ge -\frac{F^{pq,rs}\tilde{g}_{p \bar{q}1 } \tilde{g}_{r\bar{s} \bar{1} } }{\lambda_1} 
+ \frac{1}{32K}F^{pp} \lambda_p^2 -26A^2 \mathcal{F}K +\frac{1}{2}\chi'^2|v_p|^2  \\
+& \Psi^{\prime \prime} F^{pp}|u_p|^2 +\Psi^{\prime}F^{pp} u_{p \bar{p}}  
 -  \sum_k \frac{1}{2A_1K^2} \left|C_4 C_2 u_{\bar{k}}\right|^2 \Tr_{\omega}(\omega_{u}) -C_0\mathcal{F}.  
\end{aligned}    
\end{equation}
We remark that the fifth term above is positive so it can be aborbed. We also have
\[
-\sum_k \frac{1}{2A_1 K^2} \left|C_4 C_2 u_{\bar{k}}\right|^2 \Tr_{\omega}(\omega_{u}) \ge - \frac{1}{2A_1} \left(C_4 C_2 \right)^2  n \lambda_1 \mathcal{F}.
\]
Using the concavity of $F^{pq,rs}$, combining the computations above, we finally obtain
\begin{equation}
   \begin{aligned}
    0 \ge F^{pp} G_{p\bar{p} } \ge \Psi^{\prime}F^{pp} u_{p \bar{p}} + \frac{1}{32K}F^{pp} \lambda_p^2 -C_0 \mathcal{F} -28A^2\mathcal{F}K
 - \frac{1}{2A_1} \left(C_4 C_2 \right)^2  n \lambda_1\mathcal{F}, 
\end{aligned}  
\end{equation}
which is similar to (108) in \cite{Sze18}, up to the coefficients of certain terms. Following the estimates (109) and (110) in \cite{Sze18}, we get the required uniform bound for $K^{-1}\lambda_1$ for this case.

\textbf{Case 2:} $-\lambda_n \le \delta \lambda_1$.
We define
\[
I = \{ p: F^{p p } > \delta^{-1 }F^{1 \bar{1} } \}.
\]
It follows from the identity $H_p = 0$ that
\[
\begin{aligned}
&-\frac{\sum_{p \notin I} F^{p p } |\tilde{g}_{1 \bar{1} p } |^2 }{\lambda_1^2 }\\
\ge& -2 (\Phi' )^2 \sum_{p \notin I } F^{p p } |\sum_k u_{kp }u_{\bar{k} } + u_{p }u_{\bar{p} p } |^2 - 6 (\Psi')^2 \sum_{p \notin I} F^{p p } |u_p |^2 - 6 (\chi' )^2 \sum_{p \notin I } F^{pp} |v_p - \tilde{\varphi}_p |^2- C_0\mathcal{F} \\
\ge& -2 (\Phi' )^2 \sum_{p \notin I } F^{p p } |\sum_k u_{kp }u_{\bar{k} } + u_{p }u_{\bar{p} p } |^2  - 6 (2A)^2 \delta^{-1} F^{11} K- 9 (\chi' )^2 \sum_{p \notin I } F^{pp } |v_p  |^2 - C_0(\mathcal{F}+ \chi'\mathcal{F}).
\end{aligned}
\] 
Commutating the derivatives as (113) in \cite{Sze18}, we obtain that for $p\notin I$:
\[
\begin{aligned}
&\frac{F^{pp }|\tilde{g}_{p\bar{1} 1 } |^2}{\lambda_1^2}\le \frac{F^{pp }|\tilde{g}_{1\bar{1} p } |^2}{\lambda_1^2} + C_0F^{pp }(1 + \lambda_1^{-1} |\tilde{g}_{1\bar{1} p } |) \\
\le& \frac{F^{pp }|\tilde{g}_{1\bar{1} p } |^2}{\lambda_1^2}+ C_0 \mathcal{F} + \frac{C_0}{2 K}F^{pp} (|u_{kp}| + |u_{p \bar{p}}| )K^{1/2} + 2A C_0 F^{pp} |u_p| + \chi' C_0 |v_p - \tilde{\varphi}_p| F^{pp}  \\
\le& \frac{F^{pp }|\tilde{g}_{1\bar{1} p } |^2}{\lambda_1^2}+ C_0 \mathcal{F} + \frac{C_0}{2 K}F^{pp} (|u_{kp}| + |u_{p \bar{p}}| )K^{1/2} + A^2 F^{pp} |u_p|^2 +   F^{pp}2\chi' |v_p|^2 + C_0 \mathcal{F}.
\end{aligned}
\]
Additionally, it is clear that
\[
- 9 (\chi' )^2 \sum_{p \notin I } F^{pp } |v_p  |^2 - \sum_{p\notin I}  F^{pp}2\chi' |v_p|^2 \ge - 9 (\chi' )^2  F^{pp } |v_p  |^2 - F^{pp}2\chi' |v_p|^2 .
\]
It then follows that 
\begin{equation} \label{eqn:FppGpp6}
    \begin{aligned}
    0 \ge& F^{pp} G_{p\bar{p} } \ge \frac{C_4}{\lambda_1}( -3C_2|v_1|^2-C_2 |1+v_{1\bar{1}}| )- \frac{F^{pq,rs}\tilde{g}_{p \bar{q}1 } \tilde{g}_{r\bar{s} \bar{1} } }{\lambda_1}- \sum_{p\in I} \frac{F^{pp }|\tilde{g}_{p\bar{1}1 } |^2}{\lambda_1^2} \\
+&\Phi^{\prime \prime}\sum_{p\in I}F^{pp}\left| u_{k p} u_{\bar{k}}+ u_p u_{\bar{p} p}\right| ^2 + \frac{1}{32K}F^{pp} \lambda_p^2  - 25 A^2 \delta^{-1} F^{1\bar{1} } K \\
+& \Psi^{\prime \prime} F^{pp}|u_p|^2 +\Psi^{\prime}F^{pp} u_{p \bar{p}} +(\chi''-9\chi'^2 -2\chi')F^{pp}|v_p|^2 + \chi'F^{pp}v_{p \bar{p} } \\
 - &\frac{1}{2} \sum_k A_1 \frac{|v_k|^2}{\Tr_{\omega}(\omega_{u})}- \frac{1}{2A_1K^2} \left|C_4 C_2\right|^2 \Tr_{\omega}(\omega_{u}) -C_0\left( \mathcal{F} + \frac{2A}{\lambda_1} F^{p\bar{p} } |u_p | + \frac{\chi'}{\lambda_1}|v_p|F^{p p }\right).  
\end{aligned}
\end{equation}
We choose $\delta $ such that
\begin{equation} \label{frml:choosedelta}
\frac{4 (\Psi')^2 \delta}{1-\delta} \le \frac{1}{2} \Psi''.
\end{equation}
Since $\psi' \ge -2A$, it hence suffices to choose $\delta$ such that $\delta \le \frac{\tau}{32A + \tau}$.
We deduce that
\begin{equation} \label{eqn:Fp=0inI}
\begin{aligned}
\sum_{ p \in I}2 (\Phi ' )^2F^{p\bar{p}}\left| \sum_ku_{k p} u_{\bar{k}}+ u_p u_{\bar{p} p}\right| ^2  = 2 \sum_{ p\in I} F^{p \bar{p }} \left|\frac{u_{1\bar{1}p } }{\lambda_1}  + \Psi' u_p + \chi' (v_p- \tilde{\varphi}_p) + w_p \right|^2  \\
\ge 2 \delta \sum_{p \in I } \frac{ F^{p \bar{p }}|u_{1\bar{1}p }| ^2 }{\lambda_1^2} - \frac{4\delta }{1- \delta } (\Psi')^2 \sum_{p \in I } F^{p p } |u_p |^2 - \frac{8\delta }{1- \delta } ( \chi')^2 \sum_{p \in I } F^{p p }|v_p |^2 - C_0 \mathcal{F},
\end{aligned}
\end{equation}
where we used
\[
|a+b+c+d |^2 \ge \delta|a|^2 - \frac{\delta}{1-\delta}|b+c+d|^2 \ge \delta|a|^2 - \frac{2\delta}{1-\delta}|b|^2- \frac{4\delta}{1-\delta}|c|^2 - \frac{4\delta}{1-\delta}|d|^2.
\]
We remark that the second term of \eqref{eqn:Fp=0inI} can be absorbed by half of the positive $\psi'' F^{pp}|u_p|^2$ term, as (115) and (120) in \cite{Sze18}. We now turn to deal with the terms involving $v_p$ and $v_{p\bar{p}}$. Note that
\[
  - \frac{8\delta }{1- \delta } ( \chi')^2 \sum_{p \in I } F^{p p }|v_p |^2 \ge - \frac{8\delta }{1- \delta } ( \chi')^2 F^{p p }|v_p |^2.
\]
Moreover, our choice of $\delta$ in \eqref{frml:choosedelta} yields that $\delta \le \frac{1}{32}A^{-1}$, and therefore, 
\[
\frac{8\delta }{1- \delta } \le \frac{1}{2}.
\]
Note that with the choice of $\chi$ in case 1, all the terms involving $v_p$ and $v_{p\bar{p}}$ vanish, except for the negative $ - \frac{8\delta }{1- \delta } ( \chi')^2 \sum_{p \in I } F^{p p }|v_p |^2$ term arising in \eqref{eqn:Fp=0inI}, and this extra negative term can be absorbed by the additional positive $\frac{1}{2}\chi'^2$ term in the choose of $\chi$. We therefore get that 
\begin{equation} \label{eqn:FppGpp7}
    \begin{aligned}
    0 \ge& F^{pp} G_{p\bar{p} } \ge- \frac{F^{pq,rs}\tilde{g}_{p \bar{q}1 } \tilde{g}_{r\bar{s} \bar{1} } }{\lambda_1}-(1-2\delta) \sum_{p\in I} \frac{F^{pp }|\tilde{g}_{p\bar{1}1 } |^2}{\lambda_1^2} \\
+&  \frac{1}{32K}F^{pp} \lambda_p^2  - 25 A^2 \delta^{-1} F^{1\bar{1} } K +\Psi^{\prime}F^{pp} u_{p \bar{p}} + \left(\frac{1}{2} \Psi' -\delta^2 A^2 - C_0A^2 \lambda_1^{-1} \right) F^{pp} |u_p|^2  \\
 - & \frac{1}{2A_1K} \left(C_4 C_2 \right)^2 \Tr_{\omega}(\omega_{u}) -C_0 \mathcal{F} .  
\end{aligned}
\end{equation}
To finish, we observe that
$$
\begin{aligned}
\sum_{p=1}^n F^{p p} \lambda_p^2&=\frac{1}{m}\left[S_m(\lambda)\right]^{\frac{1}{m} -1} \sum_{p=1}^n S_{m-1}(\lambda ;p) \lambda_p^2\\
& = \frac{1}{m}\left[S_m(\lambda)\right]^{\frac{1}{m}-1} \sum_p \lambda_p \left(S_m - S_m(\lambda ;p)\right)\\
& = \frac{1}{m}\left[S_m(\lambda)\right]^{\frac{1}{m}} S_1 -  \frac{1}{m}\left[S_m(\lambda)\right]^{\frac{1}{m}-1}(m+1) S_{m+1},
\end{aligned}
$$
where the second equality follows from the splitting formula \eqref{for:split}. We can use half of the $F^{pp} \lambda_p^2$ terms to cancel the negative $\Tr_{\omega}(\omega_{u})$ terms. More precisely, we now choose 
\[
A_1 = 32 C_4^2 C_2^2/C_3.
\]
Then we obtain 
\begin{equation} \label{eqn:FppGpp8}
    \begin{aligned}
    0 \ge& 
  \frac{1}{64K}F^{pp} \lambda_p^2  - 25 A^2 \delta^{-1} F^{1\bar{1} } K -2A(S_m^{1/m } -\mathcal{F})  + (\frac{1}{2} \Psi'' -\delta^2 A^2 - C_0A^2 \lambda_1^{-1} ) F^{pp} |u_p|^2  \\
 - & C_0 S_{m+1 }-C_0 \mathcal{F} .  
\end{aligned}
\end{equation}
We choose $A$ such that $2A \ge C_0$, and the last negative term can be absorbed. Moreover we observe that 
\[
F^{pp} \lambda_p^2 \ge F^{11} \lambda_1^2.
\]
Therefore, we assume that 
\[
\lambda_1^2 \ge 25 \cdot 128 A^2 \delta^{ -1} K^2,
\]
and we can thus use half of the first term to absorb the second term. It follows that 
\[
0 \ge \frac{1}{128K} F^{pp } \lambda_p^2 - C_0 - C_0 S_{m+1} + \left(\frac{1}{2} \Psi'' -\delta^2 A^2 - C_0A^2 \lambda_1^{-1} \right) F^{pp }|u_p |^2.
\]
To deal with the last term, we recall that $\delta$ is chosen to be small such that $\delta \le \frac{1}{16A}$, and we assume moreover that $\lambda_1 \ge 4 \tau^{-1 } A C_0$. Then the third term can be absorbed. We finally arrive at 
\[
F^{pp }  \lambda_p^2 -C_0 S_{m+1} K\le 128 C_0' K.
\]
It follows that
\[
\left[S_m(\lambda)\right]^{\frac{1}{m}} S_1 -  \left[S_m(\lambda)\right]^{\frac{1}{m}-1}(m+1) S_{m+1} - C_0 S_{m+1}K \le 128 C_0' K.
\]
If $S_{m+1 } \le0$ then we get the upper bound for $S_1$. Otherwise, using Maclaurin's inequality (Proposition~\ref{prop:espstandard}), we deduce that 
\[
\begin{aligned}
128 C_0' K \ge &\left[S_m(\lambda)\right]^{\frac{1}{m}} S_1 - C_0 S_{m+1} K \\
\ge & \left[S_m(\lambda)\right]^{\frac{1}{m}} S_1 - C_0 S_m^{\frac{m+1}{m} } K\\
\ge &
\left[S_m(\lambda)\right]^{\frac{1}{m}} S_1 -C_0''K,
\end{aligned}
\]
and we thus finish the proof.
\end{proof}

\section{Regularity of solutions} \label{sec:mainthm}
The $\mathcal{C}^2$-estimate that we obtain in the last session allows us to obtain higher order estimates. As an application of the $\mathcal{C}^2$-estimate, we prove a regularity result in this session.
\begin{theorem}
   Let $(X,\omega)$ be a compact Hemitian manifold. Let $\varphi$ be a bounded $(\omega,m )$-subharmonic function solving the complex Hessian equation
   \[
   (\omega + \ddc \varphi )^m \wedge \omega^{n-m} = e^{f(\cdot, \varphi) } \omega^n,
   \]
   where $f: X \times \R \rightarrow \R$ is a smooth function. Then $\varphi$ is smooth.
\end{theorem}
\begin{proof}
     We take a sequence of decreasing, smooth $(\omega,m)$-subharmonic functions $v_j$ converging to $\varphi$ as $j \rightarrow + \infty$. We consider the complex Hessian equation 
     \begin{equation} \label{eqn:mainthm1}
     (\omega + \ddc u_j)^m \wedge \omega^{n-m} = e^{u_j + h(x, v_j) } \omega^n,
     \end{equation}
     where $h(x,t) = f(x,t)-t$ is a smooth function. The existence of smooth solution follows essentially from \cite{Sze18}, with appropriate modifications presented in \cite[Theorem 3.17]{KN15weaksol}.

     Since $\tilde{\varphi}$ is chosen such that $|\tilde{\varphi} - v_j | <\epsilon$ and $v_j \rightarrow \varphi$, where $\varphi$ is continuous, it follows that $|v_j - \tilde{\varphi}| <2\epsilon$, where $\epsilon$ is independent of $j$.
     
     We can deduce from Theorem~\ref{thm:c2estmate} and Dinew--Ko{\l}odziej's Liouville type theorem \cite{LiouvilleDK17} that $ \sup_{X}|\Delta u_j|$ is uniformly bounded. By extracting a subsequence, we can suppose that $u_j$ converges uniformly to $u$. It follows from the continuity of the Hessian operator (see \cite{kolodziej2023complexhessianmeasuresrespect}) that $u$ satisfies
     \[
     (\omega + \ddc u )^m \wedge \omega^{n-m} = e^{u+ h(x,\varphi)} \omega^n .
     \]
     We remark that $\varphi$ is a $(\omega,m)$-sh function satisfying
     \[
     ( \omega + \ddc \varphi)^m \wedge \omega^{n-m} = e^{f (x, \varphi) } \omega^n .
     \]
     Hence it follows from the domination principle that $\varphi =u$. We deduce that \[
      |\Delta u| \le C ,
     \]
     for some constant $C$ independent of $j$.

We would like to apply \cite{TWWY_2014} but at this stage we only have a uniform bound for the $L^{\infty}$-norm of $v_j$. 

 To obtain the desired regularity for $\varphi$, we next consider the following equation, where we replace $v_j$ by $u_j$ itself:
     \begin{equation} \label{eqn:mainthmHessian2}
     (\omega +\ddc \tilde{u}_j )^m \wedge \omega^{n-m } = e^{ \tilde{u}_j - u_j + f(\cdot,u_j )}\omega^n.
     \end{equation}
     We remark that $ f(\cdot,t )$ is smooth and $u_j$ is a smooth $(\omega,m)$-subharmonic function. Hence it follows from Theorem~\ref{thm:c2estmate} that 
     \[
     |\Delta \tilde{u}_j |\le C(1+ \sup_X|\nabla \tilde{u}_j |^2 ).
     \]
     Additionally, due to Dinew--Ko{\l}odziej's Liouville type theorem \cite{LiouvilleDK17}, the gradient of $\varphi$ is bounded. Extract a subsequence and suppose that $\tilde{u}_j$ converges uniformly to $\tilde{u}$. We deduce from the continuity of the Hessian operator \cite{kolodziej2023complexhessianmeasuresrespect} that 
     \[
     (\omega + \ddc \tilde{u} )^m \wedge \omega^{n-m } = e^{\tilde{u} -u + f(x, u)}\omega^n.
     \]
     The domination principle yields that $\tilde{u}=u$ (see Corollary~\ref{coro:dp2}). Moreover, since $|\Delta u_j|$ is bounded, we deduce that $e^{ \tilde{u}_j - u_j + f(\cdot,u_j )}$ is bounded in $\mathcal{C}^{0,\alpha_0} $, for some $0<\alpha_0 <1$. It thus follows from \cite{TWWY_2014} that solutions to equation \eqref{eqn:mainthmHessian2} satisfy $\|\tilde{u}_j \|_{\mathcal{C}^{2, \alpha } } \le C_1$. We conclude that $\tilde{u} \in \mathcal{C}^{2,\alpha'}$ for all $\alpha' < \alpha$. By Schauder's estimate (see \cite[Theorem 2.10]{Szekelyhidi2014book}), we obtain that $\tilde{u}$ is smooth, and so is $\varphi$. 
     \end{proof}

\bibliographystyle{alpha}
\bibliography{ref}
\end{document}